\documentclass[10pt,oneside,final]{amsart}
\usepackage[T1]{fontenc}
\usepackage[latin9]{inputenc}
\usepackage{calrsfs}
\usepackage{geometry}
\usepackage{amssymb}
\usepackage{esint}

	\theoremstyle{definition}
	\newtheorem{defn}{Definition}[section]

	\theoremstyle{plain}
	\newtheorem{thm}[defn]{Theorem}
	\newtheorem*{thm*}{Theorem}
	\newtheorem{lem}[defn]{Lemma}
	\newtheorem*{lem*}{Lemma}
	\newtheorem{prop}[defn]{Proposition}
	\newtheorem*{prop*}{Proposition}
	\newtheorem{cor}[defn]{Corollary}
	\newtheorem*{cor*}{Corollary}
	\theoremstyle{remark}
	\newtheorem{rmk}[defn]{}
	
	\DeclareMathOperator{\Ker}{Ker}
	\newcommand{\ms}{\text{ }}

	\newcommand{\R}{\mathbb R}
	\newcommand{\C}{\mathbb C}

	\newcommand{\mb}[1]{\mathbb{#1}}
	
	\newcommand{\mc}[1]{\mathcal{#1}}
	
	\newcommand{\frk}[1]{\mathfrak{#1}}
	
	\newcommand{\norm}[1]{\left\| #1 \right\|}
	\newcommand{\vnorm}{\norm{\hspace{.15cm}}}
	\newcommand{\projnorm}[1]{\norm{#1}^{\widehat{\ms}}}
	\newcommand{\vprojnorm}{\vnorm^{\widehat{\ms}}}

\begin{document}
\title[Analytic Subordination for Free Compression]{Analytic Subordination for Free Compression\\ \small{(Preliminary Version)}}
\author{Stephen Curran}
\address{Department of Mathematics\\University of California at Berkeley\\Berkeley, CA 94720}
\email{curransr@math.berkeley.edu}
\begin{abstract}
We extend the free difference quotient coalgebra approach to analytic subordination to the case of a free compression in free probability.
\end{abstract}

\maketitle
\section{Introduction}

If $\mu, \nu$ are Borel probability measures on $\R$, let $\mu \boxplus \nu$ denote their free additive convolution.  The Cauchy transform $G_{\mu \boxplus \nu}$ is analytically subordinate to $G_\mu$ in the upper half plane, i.e. there is an analytic function $f: \mb H_+(\C) \to \mb H_+(\C)$ such that $G_{\mu \boxplus \nu}(z) = G_\mu(f(z))$.  This result is the main tool in proving regularity properties of free convolution, and was first proved by D. Voiculescu in \cite{entropyi} under an easily removed genericity condition.  Using combinatorial methods, P. Biane showed in \cite{biane} that the subordination is an operator valued phenomenon.  Namely if $X,Y$ are self-adjoint and free random variables in a von Neumann algebra with faithful normal trace state, then the functions $(X-zI)^{-1}$ and $E_{W^*(X)}((X+Y)-zI)^{-1}$ satisfy an analytic subordination relation in the upper half plane.  In \cite{coalgebra}, it was shown that the subordination is due to a certain conditional expectation which is a coalgebra morphism between the free difference quotient coalgebras of $\partial_{X+Y}$ and $\partial_{X}$.  This approach extends to the $B$-valued case, i.e. when $X$ and $Y$ are self-adjoint elements in a von Neumann algebra with faithful normal trace state which are $B$-free, where $1 \in B$ is a W$^*$-subalgebra. 

  In \cite{semigroup}, A. Nica and R. Speicher showed that for any Borel probability measure $\mu$ on $\R$, there is a partially defined continuous free additive convolution semi-group starting at $\mu$, i.e. a continuous family $\{\mu_t: t \geq 1\}$ such that $\mu = \mu_1$, $\mu_{s+t} = \mu_s \boxplus \mu_t$.  In \cite{atoms}, S. T. Belinschi and H. Bercovici showed that the analytic subordination for $\mu^{\boxplus n}$ extends to $\mu_t$.  This can be used to prove certain regularity results for the free additive convolution semigroup.  Here we present a proof of this result following the free difference quotient coalgebra approach, which allows a $B$-valued extension.

If $X$ is a self-adjoint element in a W$^*$-probability space $(M,\tau)$ with distribution $\mu$, and $p$ is a projection in $M$ free from $X$ with $\tau(p) = t^{-1}$, then $\mu_t$ is the distribution of $t pXp$ in $(pMp,t\tau|_{pMp})$.  Here we show that a certain rescaled conditional expectation is a coalgebra morphism between the free difference quotient coalgebras of $\partial_{tpXp}$ and $\partial_X$.  We then follow the approach of (\cite{coalgebra}) to establish the subordination result.

\bigskip
\noindent\textit{Acknowledgments.}  I would like to thank Dan-Virgil Voiculescu for suggesting this problem, and for the helpful discussions and guidance while completing this paper.
\section{Preliminaries}

\begin{rmk}  Free difference quotient derivation

\medskip\noindent If $B$ is a unital algebra over $\C$ and $X$ is algebraically free from $B$, the free difference quotient is the derivation
\begin{equation*}
 \partial_{X:B} : B\langle X \rangle \to B\langle X \rangle \otimes B \langle X \rangle
\end{equation*}
which takes $B$ to $0$ and $X$ to $1 \otimes 1$.  $\partial_{X:B}$ is a coassociative comultiplication, i.e.
\begin{equation*}
 (\partial_{X:B} \otimes \text{id}) \circ \partial_{X:B} = (\text{id} \otimes \partial_{X:B}) \circ \partial_{X:B}
\end{equation*}

\end{rmk}

\begin{rmk}\label{coreps}  Corepresentations of derivation-comultiplications.

\medskip\noindent Suppose $A$ is a unital algebra over $\C$ and that $\partial:A \to A \otimes A$ is a coassociative comultiplication which is a derivation with respect to obvious $A$-bimodule structure on $A \otimes A$.  A corepresentation of $(A,\partial)$ is a $n \times n$ matrix $(a_{ij})_{1 \leq i,j \leq n}$ with entries in $A$ such that
\begin{equation*}
 \partial a_{ij} = \sum_{1 \leq k \leq n} a_{ik} \otimes a_{kj}
\end{equation*}
We recall the following characterization of invertible corepresentations from (\cite[Proposition 1.4]{coalgebra}).

\begin{prop*}
Let $(A,\partial)$ be as above, and suppose that $X \in A$ is such that $\partial(X) = 1 \otimes 1$.  If $\alpha = (a_{ij})_{1 \leq i,j \leq n}$ is a corepresentation of $(A,\partial)$, such that $\alpha$ is invertible in $\frk M_n(A)$, then
\begin{equation*}
 \alpha = ((n_{ij} - X\delta_{ij})_{1 \leq i,j \leq n})^{-1}
\end{equation*}
for some $n_{ij} \in N = \Ker \partial$.  Conversely, if $n_{ij} \in N = \Ker \partial$ are such that the matrix $\beta = (n_{ij} - X\delta_{ij})_{1 \leq i,j \leq n}$ is invertible in $\frk M_n(A)$, then $\alpha = \beta^{-1}$ is a corepresentation of $(A,\partial)$.
\end{prop*}
\qed
\end{rmk}

\begin{rmk}  Conjugate variables

\medskip\noindent  If $M$ is a von Neumann algebra with faithful normal trace state $\tau$, $1 \in B \subset M$ is a W$^*$-subalgebra, and $X = X^* \in M$ is algebraically free from $B$, then the conjugate variable $\mc J(X:B)$ is defined as the unique (if it exists) element in $L^1(W^*(B\langle X\rangle))$ such that
\begin{equation*}
 \tau(\mc J(X:B)m) = (\tau \otimes \tau)(\partial_{X:B}m) \ms \ms \ms m \in B\langle X \rangle
\end{equation*}
If $\vert \mc J(X:B) \vert_2 < \infty$, then viewing $\partial_{X:B}$ as a densely defined unbounded operator from $L^2(W^*(B \langle X \rangle)) \to L^2(W^*(B \langle X \rangle)) \otimes L^2(W^*(B \langle X \rangle))$ we have $1 \otimes 1 \in \frk D(\partial_{X:B}^*)$, $\mc J(X:B) = \partial_{X:B}^*(1 \otimes 1)$ and $\partial_{X:B}$ is closable, in particular it is closable in $\vnorm$.  (See \cite{fisher}).
\end{rmk}

\begin{rmk} Noncommutative power series

\medskip
\noindent If $K$ is a C$^*$-algebra, $A$ is a C$^*$-subalgebra and $R > 0$ then $A_R\{t\}$ (see \cite{fisher}) will denote the completion of the ring $A\langle t\rangle$ of noncommutative polynomials with coefficients in $A$ with respect to the norm $\vert\ms\vert_R$ defined by
\begin{equation*}
 \vert P \vert_R = \inf \sum_{k \in \mb N} \norm{a_1^{(k)}}\dotsb\norm{a_{n(k)}^{(k)}}R^{n(k)-1}
\end{equation*}
where the infimum is taken over all representations of a noncommutative polynomial $P \in A\langle t \rangle$ as a sum with finite support of the form
\begin{equation*}
 P(t) = \sum_{k \in \mb N} a_1^{(k)}ta_2^{(k)}t\dotsb a_{n(k)}^{(k)}
\end{equation*}
If $X \in K$ and $\norm X < R$, then $f(X)$ is well defined for any $f(t) \in A_R\{t\}$ and $\norm{f(X)} \leq \vert f\vert_R$.  
\end{rmk}

\begin{rmk} \label{half} Half-planes of a C$^*$-algebra

\medskip\noindent If $K$ is a C$^*$-algebra, we define the upper and lower half-planes in $K$ by
\begin{align*}
 \mb H_+(K) &= \{T \in K| \text{Im }T \geq \epsilon 1 \text{ for some }\epsilon > 0\}\\
\mb H_-(K) &= \{T \in K| \text{Im }T \leq -\epsilon 1 \text{ for some }\epsilon > 0\}
\end{align*}
If $T \in \mb H_+(K)$, then $T$ is invertible, and
\begin{align*}
 \norm{T^{-1}} &\leq \epsilon^{-1} & -(\epsilon + \epsilon^{-1}\norm{T}^2)^{-1} \geq \text{Im }T^{-1}
\end{align*}
In particular, $T^{-1} \in \mb H_-(K)$. (See \cite[3.6]{coalgebra}).

\medskip
\noindent By $\Delta_+ \frk M_n(K)$ we will denote the set of matrices $\kappa = (k_{ij})_{1 \leq i,j \leq n} \in \frk M_n(K)$ such that $k_{ii} \in \mb H_+(K), 1 \leq i \leq n$, and $k_{ij} = 0$ for $i < j$.  So $\Delta_+\frk M_n(K)$ is the set of lower triangular matrices with diagonal entries in $\mb H_+(K)$.  Let $\Delta_-\frk M_n(K)$ denote the lower triangular matrices with diagonal entries in $\mb H_-(K)$.  Note that if $\kappa \in \Delta_\pm\frk M_n(K)$ then $\kappa^{-1} \in \Delta_{\mp}\frk M_n(K)$ and $(\kappa^{-1})_{ii} = (\kappa_{ii})^{-1}$ for $1 \leq i \leq n$.
\end{rmk}

\section{The coalgebra morphism associated to free compression}

\rmk In this section we study a certain rescaled conditional expectation, which for a free compression gives a coalgebra morphism between free difference quotient coalgebras.  The framework is $(M,\tau)$, a von Neumann algebra with a faithful normal trace state.  If $A,B$ are subalgebras in $M$, $A \vee B$ will denote the subalgebra generated (algebraically) by $A \cup B$.  If $1 \in A \subset M$ is a $*$-subalgebra, $E^{(M)}_A$ will denote the unique conditional expectation of $M$ onto $W^*(A)$ which preserves $\tau$.  If $p \in M$ is a projection in $M$, $\tau_{p}$ will denote the faithful normal trace state on $pMp$ given by $\tau_{p} = \tau(p)^{-1}\tau|_{pMp}$.

\begin{lem} \label{cmorph}
Suppose that $1 \in B \subset M$ is a $*$-subalgebra, $X = X^* \in M$ and that $p \in M$ is a projection such that $p$ commutes with $B$ and $X$ is algebraically free from $B[p]$.  Let $\alpha$ denote $\tau(p)$, and put $X_p = \alpha^{-1}pXp$, which we consider as a $Bp$-valued random variable in $pMp$.  Define $\psi:pMp \to M$ by $\psi(pmp) = \alpha^{-1}pmp$.  Then $\psi(Bp\langle X_p \rangle) \subset B\langle p,X \rangle$ and 
\begin{equation*}
 (\psi \otimes \psi) \circ \partial_{X_p:Bp} = \partial_{X:B[p]} \circ \psi|_{Bp \langle X_p \rangle}
\end{equation*}
i.e., $\psi|_{B_p\langle X_p\rangle }$ is a coalgebra morphism for the comultiplications $\partial_{X_p:Bp}$ and $\partial_{X:B[p]}$.
\end{lem}

\begin{proof}
Clearly $\psi(Bp\langle X_p\rangle) \subset B\langle p,X\rangle$, we must show that $\psi$ is comultiplicative.  Both sides of the above equation are derivations from $Bp\langle X_p\rangle$ into $M \otimes M$ with respect to the natural $Bp\langle X_p \rangle$ bimodule structure on $M \otimes M$.  It is clear that $Bp$ is in the kernel of both derivations, we need only compare them on $X_p$.  We have
\begin{equation*}
 \partial_{X:B[p]} \circ \psi (X_p) = \alpha^{-2}\partial_{X:B[p]}(pXp) 
= \alpha^{-2}p \otimes p
= (\psi \otimes \psi)(p \otimes p)
= (\psi \otimes \psi) \circ \partial_{X_p:Bp}(X_p)
\end{equation*}

\end{proof}

\begin{rmk} \label{CondDer}
Certain conditional expectations behave well with respect to freeness and derivations, which allows us to extend the coalgebra morphism $\psi$ to a rescaled conditional expectation.  We will need the following result from (\cite[Lemma 2.2]{coalgebra}).

\begin{lem*}
Let $1 \in B$ be a $W^*$-subalgebra, and let $1 \in A, 1 \in C$ be $*$-subalgebras in $(M,\tau)$.  Assume $A$ and $C$ are $B$-free in $(M,E_B)$.  Let $D: A \vee B \vee C \to (A \vee B \vee C) \otimes (A \vee B \vee C)$ be a derivation such that $D(B \vee C) = 0$ and $D(A \vee B) \subset (A \vee B) \otimes (A \vee B)$.  Then
\begin{equation*}
 (E_{A \vee B} \otimes E_{A \vee B}) \circ D = D \circ E_{A \vee B}|_{A \vee B \vee C}
\end{equation*}
\end{lem*}
\qed
\end{rmk}

\begin{prop}\label{expmorph}
Suppose that $1 \in B \subset M$ is a W$^*$-subalgebra, $X = X^* \in M$ and that $p \in M$ is a projection such that $p$  is $B$-free with $X$, $p$ commutes with $B$ and $X$ is algebraically free from $B[p]$.  Let $\alpha$ denote $\tau(p)$, and put $X_p = \alpha^{-1}pXp$.  Define $\Psi:pMp \to M$ by $\Psi = E^{(M)}_{B\langle X\rangle} \circ \psi$.  Then
\begin{equation*}
 (\Psi \otimes \Psi) \circ \partial_{X_p:Bp} = \partial_{X:B}\circ \Psi|_{Bp\langle X_p \rangle}
\end{equation*}

\end{prop}

\begin{proof}
 Since $X$ and $p$ are $B$-free in $M$, $E^{(M)}_{B \langle X \rangle }B[X,p] \subset B \langle X \rangle $ so that
\begin{equation*}
 \Psi(Bp[X_p]) \subset B \langle X \rangle 
\end{equation*}
By the previous lemma applied to $A = \C[X], B = B, C = \C[p], D = \partial_{X:B[p]}$ we have
\begin{equation*}
 \left(E^{(M)}_{B\langle X\rangle} \otimes E^{(M)}_{B\langle X\rangle}\right) \circ \partial_{X:B[p]} = \partial_{X:B} \circ E^{(M)}_{B\langle X \rangle]}\Big\vert_{B\langle X,p\rangle }
\end{equation*}
The result then follows from composing both sides with $\psi|_{Bp[X_p]}$ and applying Lemma \ref{cmorph}.
\end{proof}

\begin{rmk}
To attach probabilistic meaning to the map $\Psi$, it should be unital and preserve trace and expectation onto $B$.  These properties require the additional assumption that $p$ is independent from $B$ with respect to $\tau$.
\begin{prop*}
Let $M,B,X,p,\Psi$ as above and suppose, in addition to the previous hypotheses, that $p$ is independent from $B$ with respect to $\tau$.  Then $\Psi(bp) = b$ for $b \in B$, in particular $\Psi$ is unital.  Furthermore, $\Psi$ preserves trace and expectation onto $B$, i.e.
\begin{align*}
\tau \circ \Psi &= \tau_{p}\\
\Psi \circ E_{Bp}^{(pMp)} &= E_B^{(M)} \circ \Psi
\end{align*}
\end{prop*}
 
\end{rmk}

\begin{proof}
First remark that independence implies $E^{(M)}_B(p) = \alpha$.  Since $X$ and $p$ are $B$-free,
\begin{equation*}
E^{(M)}_{B \langle X \rangle }(p) = E^{(M)}_B(p) = \alpha
\end{equation*}
Therefore, for $b \in B$ we have
\begin{equation*}
 \Psi(bp) = \alpha^{-1}E^{(M)}_{B \langle X \rangle }(bp) = \alpha^{-1}bE^{(M)}_{B \langle X \rangle }(p) = b
\end{equation*}
Next observe that
\begin{align*}
\tau\left(\Psi(pmp)\right) &= \alpha^{-1}\tau\left(E^{(M)}_{B \langle X \rangle }(pmp)\right)\\
&= \alpha^{-1}\tau(pmp)\\
&= \tau_{p}(pmp)
\end{align*}
so that $\Psi$ preserves trace.  Next we claim that
\begin{equation*}
 E_{Bp}^{(pMp)}(pmp) = \alpha^{-1}E_B^{(M)}(pmp)p 
\end{equation*}
First observe that the right hand side is a conditional expectation from $pMp$ onto $W^*(Bp)$.  Since $E_{Bp}^{(pMp)}$ is the unique such conditional expectation which preserves $\tau_{p}$, it remains only to show that this map is trace preserving.  We have
\begin{equation*}
 \tau_{p}\left(\alpha^{-1}E_B^{(M)}(pmp)p\right) = \alpha^{-2}\tau\left(E_B^{(M)}(pmp)p\right)
= \alpha^{-1}\tau(pmp)
= \tau_{p}(pmp)
\end{equation*}
which proves the claim.  We then have
\begin{align*}
 \left(\Psi \circ E_{Bp}^{(pMp)}\right)(pmp) &= \Psi\left(\alpha^{-1}E_{B}^{(M)}(pmp)p\right)\\
&= \alpha^{-2}E^{(M)}_{B \langle X \rangle }\left(E_B^{(M)}(pmp)p\right)\\
&= E_B^{(M)}\left(\alpha^{-1}E^{(M)}_{B \langle X \rangle }(pmp)\right)\\
&= \left(E_B^{(M)} \circ \Psi\right)(pmp)
\end{align*}
So that $\Psi$ preserves expectation onto $B$.

\end{proof}

\begin{rmk}\label{conjugate}
If $X = X^*, Y = Y^* \in M$ are $B$-free, where $1 \in B \subset M$ is a W$^*$-subalgebra, then if $\mc J(X:B)$ exists so does $\mc J(X+Y:B)$ and is obtained from a conditional expectation.  This is also true for a free compression:
\begin{prop*}
Suppose that $1 \in B \subset M$ is a W$^*$-subalgebra, $X = X^* \in M$ and that $p \in M$ is a projection such that $p$ commutes with $B$ and $X$ is algebraically free from $B[p]$.  Let $\alpha$ denote $\tau(p)$, and put $X_p = \alpha^{-1}pXp$.  Assume that $p$ and $B$ are independent, and that $X$ and $p$ are $B$-freely independent.  If $\mc J(X:B)$ exists, then $\mc J(X_p:Bp)$ exists and is given by
\begin{equation*}
 E^{(pMp)}_{Bp\langle X_p \rangle}(p\mc J(X:B)p)
\end{equation*}
\end{prop*}

\end{rmk}

\begin{proof}
Let $\Psi$ be as above, then for $pmp \in Bp \langle X_p \rangle$ we have
\begin{align*}
 (\tau_{p} \otimes \tau_p)(\partial_{X_p:Bp}(pmp)) &= (\tau \otimes \tau)(\partial_{X:B}\Psi(pmp)) \\
&= \alpha^{-1}\tau\left(\mc J(X:B)E^{(M)}_{B \langle X\rangle}(pmp)\right)\\
&= \alpha^{-1}\tau(\mc J(X:B)pmp)\\
&= \tau_{p}((p\mc J(X:B)p)pmp)\\
&= \tau_{p}\left(E^{(pMp)}_{Bp\langle X_p\rangle}(p\mc J(X:B)p)pmp\right)
\end{align*}

\end{proof}

\section{Completely positive morphisms between free difference quotient coalgebras}
\noindent In this section we will prove the analytic subordination result for a free compression.  We will follow Voiculescu's approach in (\cite[Section 3]{coalgebra}).

\begin{rmk}\label{cloder}
We begin with a standard result on unbounded derivations on C$^*$-algebras (\cite{coalgebra},\cite{bratelli})
\begin{lem*}
Let $K,L$ be unital C$^*$-algebras, let $\varphi_1,\varphi_2:K \to L$ be unital $*$-homorphisms, let $1 \in A \subset K$ be a unital $*$-subalgebra, and let $D:A \to L$ be a closable derivation with respect to the $A$-bimodule structure on $L$ defined by $\varphi_1,\varphi_2$.  The closure $\overline{D}$ is then a derivation, and the domain of definition $\frk D(\overline{D})$ is a subalgebra.  Moreover, if $a \in A$ is invertible in $K$, then $a^{-1} \in \frk D(\overline{D})$ and
\begin{equation*}
 \overline D(a^{-1}) = -\varphi_1(a^{-1})D(a)\varphi_2(a^{-1})
\end{equation*}

\end{lem*}
\qed
\end{rmk}

\noindent We will now restate two lemmas from \cite[Section 3]{coalgebra} in the C$^*$-context, since the proofs carry over directly we will omit them.
\begin{lem}\label{3.4} Let $K$ be a unital C$^*$-algebra and $1 \in A \subset K$ a C$^*$-subalgebra.  Suppose $X = X^* \in K$ is algebraically free from $A$, $\norm{X} \leq R$ and $\partial_{X:A}$ is closable.  If $f \in A_R\{t\}$ then $f(X) \in \frk D(\overline \partial_{X:A})$.  Moreover, if $\overline\partial_{X:A} f(X) = 0$, then $f(X) \in A$.
\end{lem}
\qed
\begin{lem}\label{3.5}
Let $K,A$ and $X$ as above, $f \in A\langle t \rangle$, and let $P = f(X)$.  Then 
\begin{equation*}
\vert f \vert_R \leq \sum_{p \geq 0} \projnorm{\partial_{X:A}^{(p)}P}_{(p+1)}(\norm X + R)^p
\end{equation*}
where $\vprojnorm_{(s)}$ is the norm on the $s$-fold projective tensor product $K^{\widehat \otimes s}$.
\end{lem}
\qed

\noindent  We are now prepared to prove a subordination result for completely positive coalgebra morphisms between free difference quotient coalgebras.  The proof follows (\cite[Proposition 3.7]{coalgebra}).

\begin{thm} \label{subord}
Let $K$ and $L$ be unital C$^*$-algebras, and $1 \in A \subset K$, $1 \in B \subset L$ C$^*$-subalgebras.  Let $X = X^* \in K$ algebraically free from $A$, $Y = Y^* \in L$ algebraically free from $B$. Suppose $\Psi:L \to K$ is a unital,  completely positive linear map such that $\Psi(B\langle Y \rangle) \subset A \langle X \rangle $ and
\begin{equation*}
(\Psi \otimes \Psi) \circ \partial_{Y:B} = \partial_{X:A} \circ \Psi|_{B\langle Y \rangle}
\end{equation*}
Suppose also that $\partial_{X:A}$ and $\partial_{Y:B}$ are closable.  Then there is a holomorphic map $F_n:\mb H_+(\frk M_n(B)) \to \mb H_+(\frk M_n(A))$ such that
\begin{equation*}
 \frk M_n(\Psi)\left((Y\otimes I_n - \beta)^{-1}\right) = (X\otimes I_n - F_n(\beta))^{-1}
\end{equation*}
for $\beta \in \mb H_+(\frk M_n(B))$.
\end{thm}

\begin{proof}
By replacing $(K,L,A,B,X,Y,\Psi)$ with $(\frk M_n(K),\frk M_n(L), \frk M_n(A), \frk M_n(B), X \otimes I_n, Y \otimes I_n, \frk M_n(\Psi))$, we may assume without loss of generality that $n=1$.  Let $\overline\partial_{X:A}$ and $\overline\partial_{Y:B}$ denote the closures of $\partial_{X:A}$ and $\partial_{Y:B}$.  We have $\Psi( \frk D(\overline \partial_{Y:B})) \subset \frk D(\overline \partial_{X:A})$ and 
\begin{equation*}
 (\Psi \otimes \Psi) \circ \overline\partial_{Y:B} = \overline\partial_{X:A} \circ \Psi|_{B\langle Y \rangle}
\end{equation*}
For $\beta \in \mb H_+(B)$,  
\begin{equation*}
(\beta - Y)^{-1} \in \frk D(\overline \partial_{Y:B})
\end{equation*}
by Proposition \ref{cloder}, and so $(\beta - Y)^{-1}$ is a corepresentation of the coalgebra $(\frk D(\overline \partial_{Y:B}),\overline \partial_{Y:B})$ by (\ref{coreps}).  Therefore $\gamma = \Psi((\beta - Y)^{-1})$ is a corepresentation of $(\frk D(\overline \partial_{X:A}),\overline \partial_{X:A})$.  Since $(\beta - Y)^{-1} \in \mb H_-(L)$ and $\Psi$ is positive and unital,  we have $\gamma \in \mb H_-(K)$ .  In particular $\gamma$ is invertible.  Note that since 
\begin{equation*}
\partial_{X:A}(a^*) = \sigma_{12}((\partial_{X:A}(a))^*) 
\end{equation*}
where $\sigma_{12}$ is the automorphism of $M \otimes M$ defined by $\sigma_{12}(m_1 \otimes m_2) = m_2 \otimes m_1$, $\frk D(\overline \partial_{X:A})$ is a $*$-algebra.  Hence $\gamma^{-1} \in \frk D(\overline\partial_{X:A})$ by Proposition \ref{cloder}.  By (\ref{coreps}),
\begin{equation*}
 \gamma^{-1} = \eta - X
\end{equation*}
for some $\eta \in \Ker \overline \partial_{X:A}$.  Since $\gamma^{-1} \in \mb H_+(K)$, we have $\eta \in \mb H_+(K)$.   Clearly the map taking $\beta \in \mb H_+(B)$ to $\eta \in \mb H_+(K)$ is a holomorphic map, it remains only to show that $\eta \in A$.  By analytic continuation, it suffices to show this for $\beta$ in an open subset of $\mb H_+(B)$.

Let $\rho = 6(\norm X + \norm Y+1)$, and put
\begin{equation*}
 \omega = \{\beta \in B \vert \norm{i\rho - \beta} < 1 \} \subset \mb H_+(B)
\end{equation*}
If $\beta \in \omega$, then
\begin{equation*}
 (\beta - Y)^{-1} = (i\rho(1 - \Gamma))^{-1} = (i\rho)^{-1}\sum_{m \geq 0} \Gamma^m
\end{equation*}
where
\begin{equation*}
 \Gamma = (i\rho)^{-1}(i\rho - \beta + Y)
\end{equation*}
Note that $\norm{\Gamma} < 1/6$.  

Let $\vprojnorm_{(p)}$ denote the projective tensor product norm on $K^{\widehat \otimes p}$.  Define $\varphi_j:K \to K^{\widehat \otimes (p+1)}$ by $\varphi_j(k) = 1^{\otimes (j-1)} \otimes k \otimes 1^{\otimes (p+1)-j}$.  Then since $\partial_{Y:B}\Gamma = (i\rho)^{-1}1 \otimes 1$, it follows easily that
\begin{equation*}
 \partial_{Y:B}^{(p)}\Gamma^m = \sum_{\substack{m_1 \geq 0,\dotsc, m_{p+1} \geq 0\\m_1 + \dotsb m_{p+1} = m-p}} (i\rho)^{-p}\varphi_1(\Gamma^{m_1})\dotsb\varphi_{p+1}(\Gamma^{m_{p+1}})
\end{equation*}
From this it follows that
\begin{equation*}
\projnorm{\partial_{Y:B}^{(p)}\Gamma^m}_{(p+1)} < \rho^{-p}6^{-(m-p)} \frac{m!}{p!(m-p)!}
\end{equation*}
if $m \geq p$, while if $m < p$ then
\begin{equation*}
 \partial_{Y:B}^{(p)}\Gamma^m = 0
\end{equation*}
Let $P_m = \Psi(\Gamma^m)$.  Then $P_m \in A \langle X \rangle $ and 
\begin{equation*}
 \partial_{X:A}^{(p)}P_m = \Psi^{\otimes (p+1)}(\partial_{Y:B}^{(p)}\Gamma^m)
\end{equation*}
Hence
\begin{equation*}
\projnorm{\partial_{X:A}^{(p)}P_m}_{(p+1)} \leq \rho^{-p}6^{-(m-p)}\frac{m!}{p!(m-p)!}
\end{equation*}
if $m \geq p$ and is zero if $m < p$.  Let $h_m \in A\langle t\rangle$ so that $P_m = h_m(X)$.  By Lemma \ref{3.5},
\begin{align*}
\vert h_m \vert_r &\leq \sum_{p \geq 0} \projnorm{\partial_{X:A}^{(p)}P_m}_{(p+1)}(\norm{X} + r)^p\\ 
&< \sum_{0 \leq p \leq m} \rho^{-p} 6^{-(m-p)}(\norm{X} + r)^p \frac{m!}{p!(m-p)!}\\
&= (\rho^{-1}(\norm X + r) + 6^{-1})^m
\end{align*}
Let $r = \norm{X}+1$, so that
\begin{equation*}
 \vert h_m \vert_r < (1/2)^m
\end{equation*}
if $m \geq 1$.  Then $h = \sum_{m \geq 1} h_m \in A_r\{t\}$ and $\vert h\vert_r < 1$.  It follows that $1 + h$ is invertible in $A_r\{t\}$.  We then have
\begin{align*}
 \eta - X &= (\Psi(\beta - Y)^{-1})^{-1}\\
&= (i\rho)\left(1 + \sum_{k \geq 1} P_k\right)^{-1}\\
&= (i\rho)(1 + h)^{-1}(X)
\end{align*}
Hence $\overline \partial_{X:A} \eta = 0$ and $\eta = g(X)$ where $g = t + (i\rho)(1 + h)^{-1} \in A_r\{t\}$ and $r = \norm X + 1$.  By Lemma \ref{3.4}, $\eta \in A$.
\end{proof}

\begin{cor}\label{subcor}
Let $(M,\tau)$ be a von Neumann algebra with faithful normal trace state, and $1 \in B \subset M$ a W$^*$-subalgebra.  Suppose $X = X^* \in M$ and that $p \in M$ is a projection which is $B$-free with $X$ and such that $p$ is independent from $B$ with respect to $\tau$.  Let $\alpha$ denote $\tau(p)$, and put $X_p = \alpha^{-1}pXp$.  Assume that $|\mc J(X:B)|_2 < \infty$.  Then there is an analytic function $F_n:\mb H_+(\frk M_n(B)) \to \mb H_+(\frk M_n(B))$ such that
\begin{equation*}
\alpha^{-1}E^{(\frk M_n(M))}_{\frk M_n(B\langle X\rangle)}\left(X_p\otimes I_n - \beta(p \otimes I_n)\right)^{-1} = (X\otimes I_n - F_n(\beta))^{-1}
\end{equation*}
for $\beta \in \mb H_+(\frk M_n(B))$.
\end{cor}

\begin{proof}
By Proposition \ref{conjugate}, also $|\mc J(X_p:Bp)|_2 < \infty$, hence $\partial_{X:B}$ and $\partial_{X_p:Bp}$ are closable in norm.  By Proposition \ref{expmorph}, Theorem \ref{subord} applies to $K = M$, $L = pMp$, $A = B$, $B = Bp$, $X = X$, $Y = X_p$, $\Psi = \Psi$ which gives the result.
\end{proof}

\begin{rmk}\label{diagcor}
In $B$-valued free probability, it is useful also to consider matricial resolvents $(X \otimes I_n - \beta)^{-1}$ where $\beta \in \Delta_+\frk M_n(B)$ (see \cite{bvalued}).  The subordination extends also to these resolvents.
\begin{cor*}
Let $K$ and $L$ be unital C$^*$-algebras, and $1 \in A \subset K$, $1 \in B \subset L$ C$^*$-subalgebras.  Let $X = X^* \in K$ algebraically free from $A$, $Y = Y^* \in L$ algebraically free from $B$. Suppose $\Psi:L \to K$ is a unital,  completely positive linear map such that $\Psi(B\langle Y \rangle) \subset A \langle X \rangle $ and
\begin{equation*}
(\Psi \otimes \Psi) \circ \partial_{Y:B} = \partial_{X:A} \circ \Psi|_{B\langle Y \rangle}
\end{equation*}
Suppose also that $\partial_{X:A}$ and $\partial_{Y:B}$ are closable.  Then there is a holomorphic map $\Phi_n: \Delta_+\frk M_n(B) \to \Delta_+\frk M_n(A)$ such that
\begin{equation*}
 \frk M_n(\Psi)\left((Y\otimes I_n - \beta)^{-1}\right) = (X\otimes I_n - \Phi_n(\beta))^{-1}
\end{equation*}
for $\beta \in \Delta_+\frk M_n(B)$.
\end{cor*}

\end{rmk}

\begin{proof}
Let $\beta = (b_{ij})_{1 \leq i,j \leq n}$, and $\gamma = (\gamma_{ij})_{1 \leq i,j \leq n}$ where 
\begin{equation*}
 \gamma = \frk M_n(\Psi)\left((Y\otimes I_n - \beta)^{-1}\right)
\end{equation*}
Then since $\beta \in \Delta_+ \frk M_n(B)$ we have $\gamma \in \Delta_+ \frk M_n(A \langle X \rangle)$ and
\begin{equation*}
 \gamma_{ii} = \Psi\left((Y - b_{ii})^{-1}\right)
\end{equation*}
By Theorem \ref{subord}, $\gamma_{ii} = (X - \eta_{ii})^{-1}$ for some $\eta_{ii} \in \mb H_+(A)$.  Since $\gamma \in \Delta_+\frk M_n(A \langle X \rangle)$, $\gamma^{-1} \in \Delta_-\frk M_n(A \langle X \rangle)$ and
\begin{equation*}
 (\gamma^{-1})_{ii} = (\gamma_{ii})^{-1} = X - \eta_{ii}
\end{equation*}
so that $\gamma^{-1} - X\otimes I_n = \eta$ for some $\eta \in \Delta_-\frk M_n(A)$.  The analytic dependence on $\beta$ is clear, so this completes the proof.
\end{proof}

\begin{cor}
Let $(M,\tau)$ be a von Neumann algebra with faithful normal trace state, and $1 \in B \subset M$ a W$^*$-subalgebra.  Suppose $X = X^* \in M$ and that $p \in M$ is a projection which is $B$-free with $X$ and such that $p$ is independent from $B$ with respect to $\tau$.  Let $\alpha$ denote $\tau(p)$, and put $X_p = \alpha^{-1}pXp$.  Assume that $|\mc J(X:B)|_2 < \infty$.  Then there is an analytic function $\Phi_n:\Delta_+\frk M_n(B) \to \Delta_+\frk M_n(B)$ such that
\begin{equation*}
\alpha^{-1}E^{(\frk M_n(M))}_{\frk M_n(B\langle X\rangle)}\left(X_p\otimes I_n - \beta(p \otimes I_n)\right)^{-1} = (X\otimes I_n - F_n(\beta))^{-1}
\end{equation*}
for $\beta \in \Delta_+\frk M_n(B)$. \qed
\end{cor}

\section{Free Markovianity for Free Compression}
\noindent We can now remove the condition on $|\mc J(X:B)|_2 < \infty$ from Corollary \ref{subcor}.  The key tool is the following ``Free Markovianity'' property of free compression.

\begin{prop}\label{markov}
Suppose that $1 \in B \subset M$ is a W$^*$-subalgebra, $X = X^* \in M$ and that $p \in M$ is a projection such that $p$ commutes with $B$, and $X$ and $p$ are $B$-freely independent.  Let $Y=Y^* \in M$ be $B$-free from $B\langle X,p\rangle$.  Then
\begin{equation*}
 E_{B\langle X\rangle}^{(M)}E_{B\langle X+Y\rangle}^{(M)}E^{(pMp)}_{Bp\langle p(X+Y)p\rangle} = E_{B\langle X\rangle}^{(M)}E^{(pMp)}_{Bp\langle p(X+Y)p\rangle}
\end{equation*}
\end{prop}

\begin{proof}
Apply \cite[Lemma 3.3]{mutual} to
\begin{align*}
D &= B, & B = W^*(B\langle X+Y\rangle)\\
A_1 &= W^*(B\langle X,Y\rangle)\\
A &= W^*(B\langle X\rangle), &\Omega = \{p\}\\
C &= W^*(B\langle X+Y,p\rangle)
\end{align*}
to conclude that $W^*(B\langle X\rangle)$, $W^*(B\langle X+Y\rangle)$, $W^*(B\langle X+Y,p\rangle)$ is freely Markovian.  By \cite[Lemma 3.7]{mutual}
\begin{equation*}
E_{B\langle X\rangle}^{(M)}E_{B\langle X+Y\rangle}^{(M)}E^{(M)}_{B\langle X+Y,p\rangle} = E_{B\langle X\rangle}^{(M)}E^{(M)}_{B\langle X+Y,p\rangle}
\end{equation*}
Since $Bp\langle p(X+Y)p\rangle \subset B\langle X+Y,p\rangle$, 
\begin{equation*}
E^{(M)}_{B\langle X+Y,p\rangle}E^{(pMp)}_{Bp\langle p(X+Y)p\rangle} = E^{(pMp)}_{Bp\langle p(X+Y)p\rangle}
\end{equation*}
from which the result follows.
\end{proof}

\noindent We are now prepared to remove the condition on $\mc J(X:B)$, this follows (\cite[Theorem 3.8]{coalgebra}).
\begin{thm}\label{mainthm}
Let $(M,\tau)$ be a von Neumann algebra with faithful normal trace state, and $1 \in B \subset M$ a W$^*$-subalgebra.  Suppose $X = X^* \in M$ and that $p \in M$ is a projection which is $B$-free with $X$ and such that $p$ is independent from $B$ with respect to $\tau$.  Let $\alpha$ denote $\tau(p)$, and put $X_p = \alpha^{-1}pXp$.  Then there is an analytic function $F_n:\mb H_+(\frk M_n(B)) \to \mb H_+(\frk M_n(B))$ such that
\begin{equation*}
\alpha^{-1}E^{(\frk M_n(M))}_{\frk M_n(B \langle X \rangle)}\left(X_p\otimes I_n - \beta(p \otimes I_n)\right)^{-1} = (X\otimes I_n - F_n(\beta))^{-1}
\end{equation*}
for $\beta \in \mb H_+(\frk M_n(B))$.
\end{thm}

\begin{proof}
The analytic dependence on $\beta$ is clear, so we must prove that
\begin{equation*}
  \alpha^{-1}E^{(\frk M_n(M))}_{\frk M_n(B \langle X \rangle)}(X_p \otimes I_n - \beta(p \otimes I_n))^{-1} = (X\otimes I_n - \eta)^{-1}
\end{equation*}
for some $\eta \in \mb H_+(\frk M_n(B))$.  Let $S$ be a $(0,1)$-semicircular element in $(M,\tau)$ which is freely independent from $B\langle X,p\rangle $.  Then $X, p$ and $S$ are $B$-free (\cite[Lemma 3.3]{mutual}).  Also $X + \epsilon S$ and $p$ are $B$-free and $|\mc J(X+\epsilon S:B)|_2 < \infty$ for $\epsilon > 0$ by \cite[Corollary 3.9]{fisher}.  So we can apply Corollary \ref{subcor} to $B,X+\epsilon S, p$, it follows that there are $\eta(\epsilon) \in \mb H_+(\frk M_n(B))$ for $0 < \epsilon \leq 1$ such that
\begin{equation*}
  \alpha^{-1}E^{(\frk M_n(M))}_{\frk M_n(B\langle X+\epsilon S\rangle )}(\alpha^{-1}p(X+\epsilon S)p \otimes I_n - \beta(p \otimes I_n))^{-1} = ((X+\epsilon S)\otimes I_n - \eta(\epsilon))^{-1}
\end{equation*}
Then
\begin{equation*}
\norm{((X+\epsilon S)\otimes I_n - \eta(\epsilon))^{-1}} \leq \norm{(\alpha^{-1}p(X+\epsilon S)p\otimes I_n - \beta(p \otimes I_n))^{-1}} \leq C_1
\end{equation*}
for some fixed constant $C_1$, and by (\ref{half}) also
\begin{equation*}
\text{Im } ((X+\epsilon S)\otimes I_n - \eta(\epsilon))^{-1} \geq C_2(1 \otimes I_n)
\end{equation*}
for some constant $C_2 > 0$.  By (\ref{half}), it follows that
\begin{align*}
 \norm{\eta(\epsilon)} &\leq C_3, & \text{Im }\eta(\epsilon) &\geq C_4(1 \otimes I_n)
\end{align*}
for some constants $C_3,C_4 > 0$.
It follows that
\begin{equation*}
\lim_{\epsilon \to 0} \norm{((X+\epsilon S)\otimes I_n - \eta(\epsilon))^{-1} - ((X\otimes I_n) - \eta(\epsilon))^{-1}} = 0
\end{equation*}
and hence
\begin{equation*}
\lim_{\epsilon \to 0} \norm{((X+\epsilon S)\otimes I_n - \eta(\epsilon))^{-1} - E_{\frk M_n(B\langle X \rangle)}^{\frk M_n(M))}((X+\epsilon S)\otimes I_n - \eta(\epsilon))^{-1}} = 0
\end{equation*}
By the previous proposition,
\begin{equation*}
E^{(M)}_{B \langle X \rangle}E^{(M)}_{B\langle X+\epsilon S\rangle }E^{(pMp)}_{Bp\langle \alpha^{-1}p(X+\epsilon S)p\rangle } = E^{(M)}_{B \langle X \rangle}E^{(pMp)}_{Bp\langle \alpha^{-1}p(X+\epsilon S)p\rangle }
\end{equation*}
so that
\begin{equation*}
 \lim_{\epsilon \to 0} \norm{((X+\epsilon S) \otimes I_n - \eta(\epsilon))^{-1} - \alpha^{-1}E_{\frk M_n(B \langle X \rangle)}^{(M_n(M))}((\alpha^{-1}p(X+\epsilon S)p)\otimes I_n - \beta(p \otimes I_n))^{-1}} = 0
\end{equation*}
But also
\begin{equation*}
 \lim_{\epsilon \to 0} \norm{((\alpha^{-1}p(X+\epsilon S)p)\otimes I_n - \beta(p \otimes I_n))^{-1} - (X_p\otimes I_n - \beta(p \otimes I_n))^{-1}} = 0
\end{equation*}
Putting these equations together, we have that
\begin{equation*}
 \alpha^{-1}E_{\frk M_n(B\langle X \rangle)}^{(\frk M_n(M))}(X_p\otimes I_n - \beta(p \otimes I_n))^{-1} = \lim_{\epsilon \to 0} (X\otimes I_n - \eta(\epsilon))^{-1}
\end{equation*}
which implies that $\eta(\epsilon)$ converges in norm to some $\eta \in \mb H_+(\frk M_n(B))$ with
\begin{equation*}
 \alpha^{-1}E_{\frk M_n(B \langle X \rangle)}^{(\frk M_n(M))}(X_p \otimes I_n - \beta(p \otimes I_n))^{-1} = (X \otimes I_n - \eta)^{-1}
\end{equation*}

\end{proof}
\begin{rmk}
The same proof as Corollary \ref{diagcor} can be used to extend Theorem \ref{mainthm} to resolvents in $\Delta_+\frk M_n(B)$.
 \begin{cor*}
Let $(M,\tau)$ be a von Neumann algebra with faithful normal trace state, and $1 \in B \subset M$ a W$^*$-subalgebra.  Suppose $X = X^* \in M$ and that $p \in M$ is a projection which is $B$-free with $X$ and such that $p$ is independent from $B$ with respect to $\tau$.  Let $\alpha$ denote $\tau(p)$, and put $X_p = \alpha^{-1}pXp$.  Then there is an analytic function $\Phi_n:\Delta_+\frk M_n(B) \to \Delta_+\frk M_n(B)$ such that
\begin{equation*}
\alpha^{-1}E^{(\frk M_n(M))}_{\frk M_n(B \langle X \rangle)}\left(X_p\otimes I_n - \beta(p \otimes I_n)\right)^{-1} = (X\otimes I_n - \Phi_n(\beta))^{-1}
\end{equation*}
for $\beta \in \Delta_+\frk M_n(B)$. \qed
\end{cor*}

\end{rmk}

\bibliographystyle{amsalpha}
\bibliography{ref}

\end{document}